\documentclass[11pt]{amsart}
\usepackage[T1]{fontenc}
\usepackage{amssymb}

\usepackage{biblatex}
\usepackage{tikz}
\usetikzlibrary{tqft}
\usetikzlibrary{arrows}

\makeatletter
\newcommand*{\da@rightarrow}{\mathchar"0\hexnumber@\symAMSa 4B }
\newcommand*{\xdashrightarrow}[2][]{%
  \mathrel{%
    \mathpalette{\da@xarrow{#1}{#2}{}\da@rightarrow{\,}{}}{}}}
\newcommand*{\da@xarrow}[7]{
\sbox0{$\ifx#7\scriptstyle\scriptscriptstyle\else\scriptstyle\fi#5#1#6\m@th$}
  \sbox2{$\ifx#7\scriptstyle\scriptscriptstyle\else\scriptstyle\fi#5#2#6\m@th$}
  \sbox4{$#7\dabar@\m@th$}
  \dimen@=\wd0 
  \ifdim\wd2 >\dimen@
    \dimen@=\wd2 
  \fi
  \count@=2 %
  \def\da@bars{\dabar@\dabar@}%
  \@whiledim\count@\wd4<\dimen@\do{
    \advance\count@\@ne \expandafter\def\expandafter\da@bars\expandafter{\da@bars\dabar@ }}
  \mathrel{#3}\mathrel{\mathop{\da@bars}\limits \ifx\\#1\\\else _{\copy0} \fi \ifx\\#2\\
   \else  ^{\copy2} \fi}   \mathrel{#4}}
\makeatother
\newcommand{\cat}{\mathcal{C}}

\newcommand{\catname}[1]{\mathbf{#1}}
\newcommand{\rel}{\catname{Rel}}

\newcommand{\id}{\mathrm{id}}
\newcommand{\relto}{\dashrightarrow}

\newcommand{\idx}{\mathbf{1}}
\newcommand{\rot}{\curvearrowleft}
\newcommand{\op}{\mathrm{op}}
\newcommand{\arrows}{\rightrightarrows}

\newcommand{\suchthat}{\mid}
\newcommand{\Z}{\mathbb{Z}}
\newcommand{\R}{\mathbb{R}}

\newtheorem{thm}{Theorem}[section]
\newtheorem{prop}[thm]{Proposition}
\newtheorem{lemma}[thm]{Lemma}

\theoremstyle{definition}
\newtheorem{definition}[thm]{Definition}

\newtheorem{remark}[thm]{Remark}

\newtheorem{example}[thm]{Example}

\newtheorem*{ack}{Acknowledgements}
\numberwithin{equation}{section}

\begin{document}

\title{Frobenius objects in the category of relations}
\author{Rajan Amit Mehta}
\author{Ruoqi Zhang}
\address{Department of Mathematics \& Statistics\\
Smith College\\
44 College Lane\\
Northampton, MA 01063}
\email{rmehta@smith.edu}
\email{rzhang89@smith.edu}

\subjclass[2010]{
18B10, %Category of relations, additive relations
18B40, %Groupoids, semigroupoids, semigroups, groups (viewed as categories)
18D35, %Structured objects in a category (group objects, etc.)
18G30, %Simplicial sets, simplicial objects (in a category)
20L05, %Groupoids (i.e. small categories in which all morphisms are isomorphisms)
57R56%Topological quantum field theories
} 
\keywords{category of relations, cohomology, Frobenius algebra, groupoid, simplicial set, topological quantum field theory}

\begin{abstract}
We give a characterization, in terms of simplicial sets, of Frobenius objects in the category of relations. This result generalizes a result of Heunen, Contreras, and Cattaneo showing that special dagger Frobenius objects in the category of relations are in correspondence with groupoids. As an additional example, we construct a Frobenius object in the category of relations whose elements are certain cohomology classes in a compact oriented Riemannian manifold.
\end{abstract}

\maketitle

\section{Introduction}
It is well-known that $2$-dimensional topological quantum field theories (TQFTs) are classified by commutative Frobenius algebras \cite{abrams}. This result can be generalized to TQFTs taking values in an arbitrary symmetric monoidal category $\cat$, and this leads one to consider the notion of a \emph{Frobenius object} in $\cat$ (see e.g.\ \cite{kock-book}). 

The purpose of this paper is to study Frobenius objects in the category of relations $\rel$. The objects of $\rel$ are sets, and morphisms from a set $X$ to a set $Y$ are subsets of $X \times Y$. Given morphisms $S \subseteq X \times Y$ and $T \subseteq Y \times Z$, the composition $T \circ S \subseteq X\times Z$ is defined as 
\[T \circ S  = \{(x,z)\suchthat (x,y) \in S \text{ and }(y,z) \in T \mbox{ for some } y \in Y\}.\]

Our motivation arises from the classical analogue of topological field theory, where the target category is the ``symplectic category''. In the symplectic category, the objects are symplectic manifolds and, in principle, the morphisms are Lagrangian relations. However, since Lagrangian relations don't always compose well, the rigorous definition of the symplectic category is more complicated \cite{ww}. $\rel$ provides a toy model for the symplectic category where these difficulties do not appear.

Our main result is that, given a Frobenius object in $\rel$, one can construct an associated simplicial set. The simplicial sets that arise in this way have certain special properties; for example, for $k \geq 1$, every $k$-simplex is completely determined by its $1$-skeleton. There are other properties that we describe in Section \ref{sec:simplicial}. The simplicial sets are also equipped with an automorphism $\hat{\alpha}$ of the set of $1$-simplices, satisfying certain compatibility conditions with the simplicial structure. We show that simplicial sets, equipped with a map $\hat{\alpha}$ as above and satisfying the specified properties, are in one-to-one correspondence with Frobenius objects in $\rel$. 

A closely related result is due to Heunen, Contreras, and Cattaneo \cite{hcc}, who showed that special dagger Frobenius objects in $\rel$ are in one-to-one correspondence with groupoids. From the point of view of TQFT, the special condition is too restrictive, since special Frobenius algebras do not distinguish between surfaces of different genus. In the context of $\rel$, the dagger condition is also restrictive. Our result can be viewed as a generalization of the result of \cite{hcc} where these conditions are removed. We discuss this relationship in Section \ref{sec:groupoids}.

In Section \ref{sec:commutative}, we consider commutative Frobenius objects in $\rel$, and we find some additional properties satisfied by the corresponding simplicial sets. 

Finally, in Section \ref{sec:cohomology}, we give a construction of a Frobenius object in $\rel$ consisting of cohomology classes in a compact oriented Riemannian manifold. These Frobenius objects never satisfy the dagger condition when the manifold has positive dimension, so they do not correspond to groupoids.

\begin{ack}
We thank Ivan Contreras for many helpful discussions. R.Z. was partially supported by the Ellen Borie Fund, via the Summer Undergraduate Research Fellowship program at Smith College. 
\end{ack}

%================
\section{Frobenius objects in $\rel$}

\subsection{Definition}
We use the notation $S: X \relto Y$ to denote a subset $S \subseteq X \times Y$, viewed as a morphism from $X$ to $Y$ in $\rel$. For any set $X$, the identity morphism $\idx: X \relto X$ is the diagonal subset $\Delta(X) \subseteq X \times X$. We use $\{\bullet\}$ to represent the set with one element. 

\begin{definition}
\label{def1}
A \emph{Frobenius object} in $\rel$ is a set $X$ equipped with morphisms
\begin{itemize}
\item $\epsilon:  X \relto \{\bullet\}$ (counit)
\item $\eta: \{\bullet\} \relto X$ (unit)
\item $\delta: X \relto  X \times X$ (comultiplication)
\item $\mu: X \times X \relto X$ (multiplication)
\end{itemize}
satisfying the following axioms:
\begin{itemize}
\item (\textbf{U}) Unit Axiom:
\[\mu \circ (\eta \times \idx) = \idx = \mu \circ (\idx \times \eta)\]
\item (\textbf{C}) Counit Axiom: 
\[(\epsilon \times \idx) \circ \delta  = \idx = (\idx \times \epsilon)  \circ \delta\]
\item (\textbf{F}) Frobenius Axiom: 
\[(\idx \times \mu) \circ (\delta \times \idx ) = \delta \circ \mu = (\mu \times \idx) \circ (\idx \times \delta)\]
\end{itemize}
\end{definition}

%====
\subsection{Graphical Calculus}

The relationship between Frobenius objects and 2-dimensional TQFT motivates the use of graphical calculus for visualizing the morphisms in Definition \ref{def1} and their compositions. The morphisms $\epsilon, \eta, \delta, \mu$ are depicted as follows:
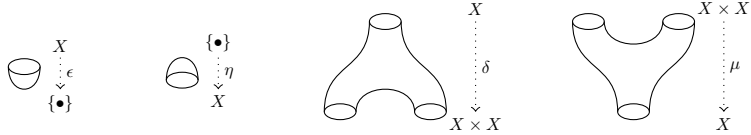
\begin{center}
\begin{tikzpicture}[scale=0.6, transform shape]
\begin{scope}[tqft/every boundary component/.style = {draw},tqft/cobordism edge/.style = {draw}]
\pic[tqft/cup, at = {(2, 0)}];
\node[above] (A) at (2.8, 0.2) {$X$};
\node[below] (B) at (2.8,-0.5) { $\{\bullet\}$}; 
\path[->] (A) edge [dotted] node[right] {$\epsilon$} (B) ;

\pic[tqft/cap, at = {(5.5, 1.7)}];
\node[above] (C) at (6.3, 0.2)  { $\{\bullet\}$};
\node[below] (D) at (6.3,-0.5)  {$X$}; 
\path[->]
(C) edge [dotted] node[right] {$\eta$} (D) ;

\pic[tqft/pair of pants, at = {(10, 1)}];
\node[above] (E) at (12, 1) { $X$};
\node[below] (F) at (12, -1)  { $X \times X$}; 
\path[->] (E) edge [dotted] node[right] {$\delta$} (F) ;

\pic[tqft/reverse pair of pants, at = {(14.5, 1)}];
\node[above] (G) at (17.5, 1) {$X \times X$};
\node[below] (H) at (17.5, -1) { $X$};
\path[->]
(G) edge [dotted] node[right] {$\mu$} (H);
\end{scope}
\end{tikzpicture}
\end{center}
We also use a cylinder to depict the identity morphism $\idx: X \relto X$. However, we will frequently omit cylinders when it is aesthetically preferable.

The Unit (\textbf{U}), Counit (\textbf{C}), and Frobenius (\textbf{F}) Axioms are depicted, respectively, as follows:
% Unit Axiom
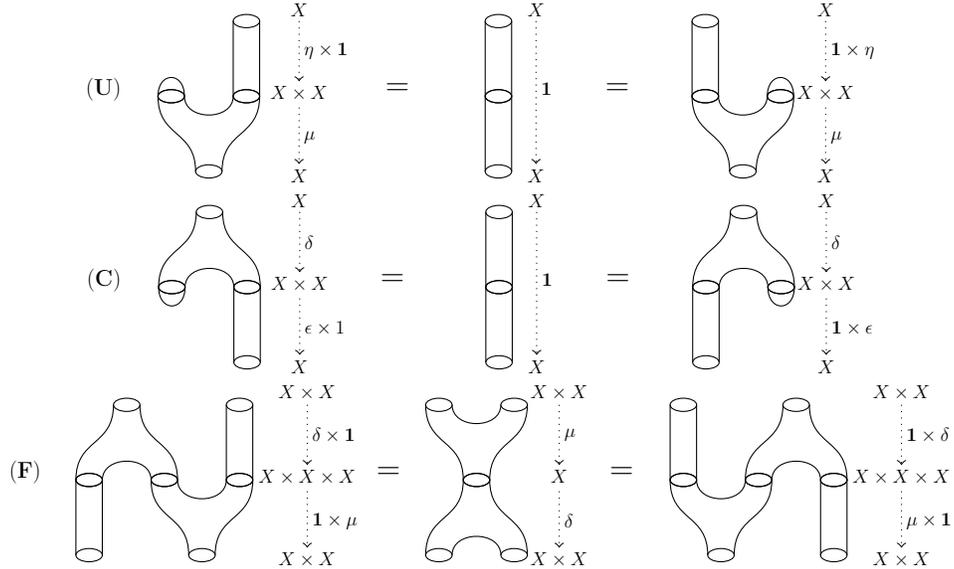
\begin{center}
\begin{tikzpicture}[scale=0.5, transform shape]
\begin{scope}[tqft/every boundary component/.style = {draw},tqft/cobordism edge/.style = {draw}]

\node at (-3,1) {\LARGE (\textbf{U})};
\pic[tqft/reverse pair of pants, name = a, at = {(-1.2, 0.8)}];
\pic[tqft/cap, anchor=outgoing boundary 1, at=(a-incoming boundary 1)];
\pic[tqft/cylinder,anchor=outgoing boundary 1, at = (a-incoming boundary 2)];

\path[draw, dotted, ->] (2.2, 2.8) node[above] {\Large $X$} --(2.2,2) node[right] {\Large $\eta \times \idx$} -- (2.2, 1.2) node[below] {\Large $X \times X$}; 

\path[draw,  dotted, ->] (2.2, 0.5) -- (2.2,-0.3) node[right] {\Large $\mu$} -- (2.2, - 1) node[below] {\Large $X$};

\node[text width = 1cm] at (5,1) {\Huge $=$};

\pic[tqft/cylinder, name = b, at = {(7.5,0.8)}];
\pic[tqft/cylinder,anchor=outgoing boundary 1, at = (b-incoming boundary 1)]; 
\path[draw, dotted,  ->] (8.5, 2.8) node[above] {\Large $X$}-- (8.5,1) node[right] {\Large $\idx$} -- (8.5, - 1) node[below] {\Large $X$};

\node[text width = 1.3cm] at (11,1) {\Huge $=$};

\pic[tqft/reverse pair of pants, name = c, at={(13,0.8)}];
\pic[tqft/cap,anchor=outgoing boundary 1, at=(c-incoming boundary 2)];
\pic[tqft/cylinder, anchor=outgoing boundary 1, at=(c-incoming boundary 1)]; 
\path[draw,  dotted, ->] (16.2, 2.8) node[above] {\Large $X$} --(16.2,2) node[right] {\Large $\idx \times \eta$} -- (16.2, 1.2) node[below] {\Large $X \times X$}; 

\path[draw, dotted,  ->] (16.2, 0.5) -- (16.2,-0.3) node[right] {\Large $\mu$} -- (16.2, - 1) node[below] {\Large $X$};
\end{scope}
\end{tikzpicture}

% Counit Axiom
\begin{tikzpicture}[scale=0.5, transform shape]
\begin{scope}[tqft/every boundary component/.style = {draw},tqft/cobordism edge/.style = {draw}]

\node at (-3,1) {\LARGE (\textbf{C})};
\pic[tqft/pair of pants, name = a, at={(-0.2,2.8)}];
\pic[tqft/cup,anchor=incoming boundary 1, at=(a-outgoing boundary 1)]; 
\pic[tqft/cylinder,anchor=incoming boundary 1, at=(a-outgoing boundary 2)];

\path[draw,dotted, ->] (2.2, 2.8) node[above] {\Large $X$} --(2.2,2) node[right] {\Large $\delta$} -- (2.2, 1.2) node[below] {\Large $X \times X$}; 

\path[draw,dotted, ->] (2.2, 0.5) -- (2.2,-0.3) node[right] {\Large $\epsilon \times 1$} -- (2.2, - 1) node[below] {\Large $X$};

\node[text width =1.3cm] at (5,1) {\Huge $=$};

\pic[tqft/cylinder, name =b, at = {(7.5,0.8)}];
\pic[tqft/cylinder,anchor=outgoing boundary 1, at = (b-incoming boundary 1)];

\path[draw,dotted, ->] (8.5, 2.8) node[above] {\Large $X$}-- (8.5,1) node[right] {\Large $\idx$} -- (8.5, - 1) node[below] {\Large $X$};

\node[text width = 1.3cm] at (11,1) {\Huge $=$};

\pic[tqft/pair of pants, name =c, at = {(14,2.8)}];
\pic[tqft/cup,anchor=incoming boundary 1, at = (c-outgoing boundary 2)];
\pic[tqft/cylinder,anchor=incoming boundary 1, at=(c-outgoing boundary 1)];
\path[draw, dotted, ->] (16.2, 2.8) node[above] {\Large $X$} --(16.2,2) node[right] {\Large $\delta$} -- (16.2, 1.2) node[below] {\Large $X \times X$}; 

\path[draw, dotted, ->] (16.2, 0.5) -- (16.2,-0.3) node[right] {\Large $\idx \times \epsilon$} -- (16.2, - 1) node[below] {\Large $X$};
\end{scope}
\end{tikzpicture}

% Frobenius Axiom
\begin{tikzpicture}[scale = 0.5, transform shape]
\begin{scope}[tqft/every boundary component/.style = {draw},tqft/cobordism edge/.style = {draw}]

\node at (-3,1) {\LARGE (\textbf{F})};
\pic[tqft/pair of pants, name = a, at = {(-0.3,2.8)}]; 
\pic[tqft/cylinder,anchor=incoming boundary 1, at = (a-outgoing boundary 1)]; 
\pic[tqft/reverse pair of pants, anchor=incoming boundary 1, name = b, at = (a-outgoing  boundary 2)]; 
\pic[,tqft/cylinder,anchor=outgoing boundary 1, at= (b-incoming boundary 2)]; 

\path[draw,dotted, ->] (4.5,2.8) node[above] {\Large $X \times X$} --(4.5,2) node[right] {\Large $\delta \times \idx$} -- (4.5, 1.2) node[below] {\Large $X \times X \times X$}; 
\path[draw,dotted, ->] (4.5, 0.5) -- (4.5,-0.3) node[right] {\Large $\idx \times \mu$} -- (4.5, - 1) node[below] {\Large $X \times X$};

\node[text width =1.3cm] at (7,1) {\Huge $=$};

\pic[tqft/reverse pair of pants, name=c, at = {(8,2.8)}]; 
\pic[tqft/pair of pants, anchor=incoming boundary 1, at =(c-outgoing boundary 1)];

\path[draw, dotted, ->] (11.2,2.8) node[above] {\Large $X \times X$} --(11.2,2) node[right] {\Large $\mu$} -- (11.2, 1.2) node[below] {\Large $X$}; 
\path[draw, dotted, ->] (11.2, 0.5) -- (11.2,-0.3) node[right] {\Large $\delta$} -- (11.2, - 1) node[below] {\Large $X \times X$};

\node[text width = 1.3cm] at (13.2,1) {\Huge $=$};

\pic[tqft/cylinder, name = d, at = {(14.5,2.8)}];
\pic[tqft/reverse pair of pants, anchor=incoming boundary 1, name = e, at=(d-outgoing boundary 1)];
\pic[tqft/pair of pants, anchor=outgoing boundary 1, name  = f, at= (e-incoming boundary 2)];
\pic[tqft/cylinder,  anchor=incoming boundary 1, at = (f-outgoing boundary 2)]; 

\path[draw,dotted,  ->] (20.3,2.8) node[above] {\Large $X \times X$} --(20.3,2) node[right] {\Large $\idx\times \delta $} -- (20.3, 1.2) node[below] {\Large $X \times X \times X$};
\path[draw,dotted, ->] (20.3, 0.5) -- (20.3,-0.3) node[right] {\Large $\mu  \times \idx $} -- (20.3, - 1) node[below] {\Large $X \times X$};
\end{scope}
\end{tikzpicture}
\end{center}

It should be emphasized that the morphisms in $\rel$ are sets; for example, $\mu$ is a subset of $X^3$. An element of the set can be depicted by labeling the boundary components. For example, $(x,y,z) \in \mu$ is depicted as follows:

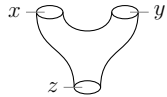
\begin{center}
\begin{tikzpicture}[scale=0.5, transform shape]
\begin{scope}[tqft/every boundary component/.style = {draw},tqft/cobordism edge/.style = {draw}]
\pic[tqft/reverse pair of pants, name = a, at = {(0,0 )}];
\node[pin=180:\LARGE $x$] at (a-incoming boundary 1) {};
\node[pin=0:\LARGE $y$] at (a-incoming boundary 2) {};
\node[pin=180:\LARGE $z$] at (a-outgoing boundary 1) {};
\end{scope}
\end{tikzpicture}
\end{center}
A diagram such as the one above can be used in two ways; it can represent a known element $(x,y,z) \in \mu$, or it can be an assertion that $(x,y,z) \in X^3$ is actually an element of $\mu$. In the latter case, we will say that the diagram is \emph{valid}.

The composition law in $\rel$ has a natural graphical interpretation, which is most easily described with an example. Consider the composition $\delta \circ \mu$, which is a subset of $X^4$. Then $(x_1,x_2,x_3,x_4) \in X^4$ is an element of $\delta \circ \mu$ if and only if there exists $y \in X$ such that $(x_1,x_2,y) \in \mu$ and $(y,x_3,x_4) \in \delta$. Graphically,
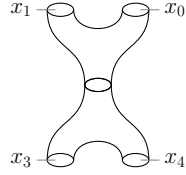
\begin{center}
\begin{tikzpicture}[scale=0.5, transform shape]
\begin{scope}[tqft/every boundary component/.style = {draw},tqft/cobordism edge/.style = {draw}]

\pic[tqft/reverse pair of pants, name =a, at= { (0,0)}]; 

\pic[tqft/pair of pants, anchor=incoming boundary 1, name = b, at = (a-outgoing boundary 1)];
\node[pin=180:\LARGE $x_1$] at (a-incoming boundary 1) {};
\node[pin=0:\LARGE $x_2$] at (a-incoming boundary 2) {};
\node[pin=180:\LARGE $x_3$] at (b-outgoing boundary 1) {};
\node[pin=0:\LARGE $x_4$] at (b-outgoing boundary 2) {};
\end{scope}
\end{tikzpicture}
\end{center}
is valid if and only if there exists $y \in X$ such that
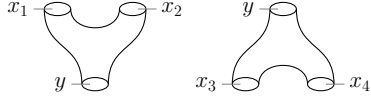
\begin{center}
\begin{tikzpicture}[scale=0.5, transform shape]
\begin{scope}[tqft/every boundary component/.style = {draw},tqft/cobordism edge/.style = {draw}]

\pic[tqft/reverse pair of pants, name = a, at = {(0,0)}];
\node[pin=180:\LARGE $x_1$] at (a-incoming boundary 1) {};
\node[pin=0:\LARGE $x_2$] at (a-incoming boundary 2) {};
\node[pin=180:\LARGE $y$] at (a-outgoing boundary 1){};

\pic[tqft/pair of pants, name = b, at = {(6,0)}];
\node[pin=180:\LARGE $y$] at (b-incoming boundary 1){};
\node[pin=180:\LARGE $x_3$] at (b-outgoing boundary 1) {};
\node[pin=0:\LARGE $x_4$] at (b-outgoing boundary 2) {};
\end{scope}
\end{tikzpicture}
\end{center}
are both valid.

%=========================
\section{Basic Properties of Frobenius objects in $\rel$}

There are several different equivalent ways to define a Frobenius object. As a result, there are various properties and structures that are only implicit in Definition \ref{def1} but that are important for understanding Frobenius objects. In this section, we review some of these properties. 

While many of the results here hold for Frobenius objects in any category (see \cite{kock-book}), our exposition emphasizes the implications for Frobenius objects in $\rel$.

\subsection{Nondegeneracy}

Let $X$ be a set. Consider a binary relation $\alpha \subseteq X \times X$, viewed as a morphism $X \times X \relto \{\bullet\}$ in $\rel$. We can depict $\alpha$ as follows:

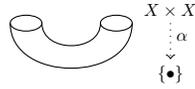
\begin{center}
\begin{tikzpicture}[scale=0.6, transform shape]
\begin{scope}[tqft/every boundary component/.style = {draw},tqft/cobordism edge/.style = {draw}]
\pic[tqft, incoming boundary components =2, outgoing boundary components =0, at = {(0,0)}];

\path[draw, dotted, ->] (3.2, 0) node[above]{$X \times X$} -- (3.2,-0.3) node[right] {$\alpha$} -- (3.2, - .8) node[below] {$\{\bullet\}$}; 
\end{scope}
\end{tikzpicture}
\end{center}

\begin{definition} 
A binary relation $\alpha$ is \emph{nondegenerate} if there exists $\beta: \{\bullet\} \relto X \times X$ such that $(\alpha \times \idx) \circ (\idx \times \beta) = (\idx \times \alpha) \circ (\beta \times \idx) = \idx$. This equation is known as the \emph{snake identity}, since it can be depicted as follows:
\end{definition}

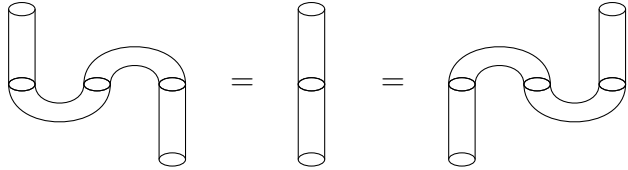
\begin{center}
\begin{tikzpicture}[scale = 0.5, transform shape]
\begin{scope}[tqft/every boundary component/.style = {draw},tqft/cobordism edge/.style = {draw}]
\pic[tqft, incoming boundary components =0, outgoing boundary components =2, name = a, at = {(-0.2,0)}];

\pic[tqft, incoming boundary components =2, outgoing boundary components =0, anchor = incoming boundary 2, name = b, at = (a-outgoing boundary 1)];

\pic[tqft/cylinder, anchor = outgoing boundary 1, at= (b-incoming boundary 1)];

\pic[tqft/cylinder, anchor = incoming boundary 1, at= (a-outgoing boundary 2)];

\node[text width = 1.3cm] at (4,-2) {\Huge $=$};

\pic[tqft/cylinder, name = c, at = {(5.5,0)}];
\pic[tqft/cylinder, anchor = incoming boundary 1, at = (c-outgoing boundary 1)];

\node[text width = 1.3cm] at (8,-2) {\Huge $=$};

\pic[tqft, incoming boundary components =0, outgoing boundary components =2, name = d, at = {(9.5,0)}];

\pic[tqft/cylinder, anchor = incoming boundary 1, at = (d-outgoing boundary 1)];

\pic[tqft, incoming boundary components =2, outgoing boundary components =0, anchor = incoming boundary 1, name = e, at = (d-outgoing boundary 2)];

\pic[tqft/cylinder, anchor = outgoing boundary 1, at = (e-incoming boundary 2)];
\end{scope}
\end{tikzpicture}
\end{center}

\begin{lemma}\label{lemma:nondegenerate}
A binary relation $\alpha$ is nondegenerate if and only if there exists a bijective map $\hat{\alpha}: X \to X$ such that $\alpha = \{(x,\hat{\alpha}(x)) \suchthat x \in X\}$. In this case, the associated relation $\beta$ is unique and given by $\beta = \{(\hat{\alpha}(x),x) \suchthat x \in X\}$.
\end{lemma}

\begin{proof}
The $(\impliedby)$ direction is immediate. We will prove the statement in the $(\implies)$ direction. 

Suppose that $\alpha$ is nondegenerate with associated relation $\beta$. Let $A,B \subseteq X^2$ be defined as
\begin{align*}
    A &= \{(x,y) \suchthat \exists z \in X \mbox{ such that } (x,z) \in \alpha, (z,y) \in \beta\}, \\
    B &= \{(x,y) \suchthat \exists z \in X \mbox{ such that } (z,x) \in \alpha, (y,z) \in \beta\}.
\end{align*}
By the snake identity, $A = B = \Delta (X)$, where $\Delta(X) \subset X \times X$ is the diagonal subset.

Since $\Delta(X) \subseteq A$, we have that for each $x \in X$, there exists a $z \in X$ such that $(x,z) \in \alpha$ and $(z,x) \in \beta$. Temporarily fix such $x$ and $z$. If $z' \in X$ is such that $(x,z') \in \alpha$, then $(z',z) \in B \subseteq \Delta(X)$, so $z=z'$. This shows that for every $x \in X$, there is a unique $z \in X$ such that $(x,z) \in \alpha$. This gives us a well-defined map $\hat{\alpha}: X \to X$ such that $\alpha = \{(x,\hat{\alpha}(x)) \suchthat x \in X\}$. We have also seen that $(\hat{\alpha}(x),x) \in \beta$ for all $x\in X$.

Assume that $\hat{\alpha}(x) = \hat{\alpha}(x')$ for $x,x' \in X$. Then $(x,x') \in A \subseteq \Delta(X)$, so $x = x'$. This shows that $\hat{\alpha}$ is one-to-one.

Let $z \in X$. Then $(z,z) \in \Delta(X) \subseteq B$, so there exists $x \in X$ such that $(x,z) \in \alpha$ and $(z,x) \in \beta$. This shows that $\hat{\alpha}$ is onto.

It remains to show that $\beta \subseteq \{(\hat{\alpha}(x),x) \suchthat x \in X\}$. Suppose that $(z,x) \in \beta$. Then, since $(x, \hat{\alpha}(x)) \in \alpha$, we have that $(\hat{\alpha}(x), z) \in B \subseteq \Delta(X)$, and therefore $z = \hat{\alpha}(x)$.
\end{proof}

\begin{lemma}\label{lemma:frobnondegenerate}
Suppose that $(X,\epsilon,\eta,\delta,\mu)$ is a Frobenius object in $\rel$, and set $\alpha := \epsilon \circ \mu$. Then $\alpha$ is nondegenerate.
\end{lemma}
\begin{proof}
Set $\beta := \delta \circ \eta$.  
Then, using Axioms (\textbf{F}), (\textbf{C}), and (\textbf{U}), we have
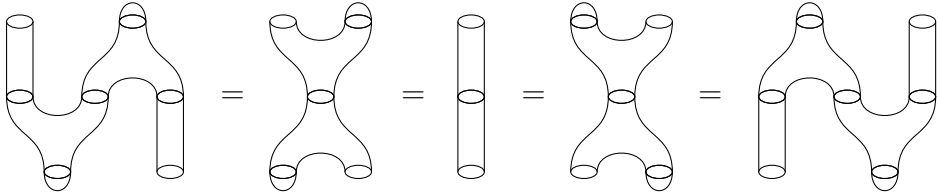
\begin{center}
\begin{tikzpicture}[scale=0.5, transform shape]
\begin{scope}[tqft/every boundary component/.style = {draw},tqft/cobordism edge/.style = {draw}]

\pic[tqft/reverse pair of pants, name = a, at = {(0,0)}];
\pic[tqft/cylinder,anchor=outgoing boundary 1, at = (a-incoming boundary 1)];
\pic[tqft/pair of pants,anchor=outgoing boundary 1, name = b, at = (a-incoming boundary 2)];
\pic[tqft/cup,anchor=incoming boundary 1, at = (a-outgoing boundary 1)];
\pic[tqft/cap,anchor=outgoing boundary 1, at= (b-incoming boundary 1)];
\pic[tqft/cylinder,anchor=incoming boundary 1, at= (b-outgoing boundary 2)];

\node[text width = 1.3cm] at (6,0) {\Huge $=$};

\pic[tqft/reverse pair of pants, name =c, at = {(7,2)}];
\pic[tqft/cap,anchor=outgoing boundary 1, at=(c-incoming boundary 2)];
\pic[tqft/pair of pants, anchor=incoming boundary 1, name = d, at = (c-outgoing boundary 1)];
\pic[tqft/cup,anchor=incoming boundary 1, at =(d-outgoing boundary 1)];

\node[text width = 1.3cm] at (10.8,0) {\Huge $=$};

\pic[tqft/cylinder, name = e, at={(12,2)}];
\pic[tqft/cylinder, anchor = incoming boundary 1, at = (e-outgoing boundary 1)];

\node[text width = 1.3cm] at (14,0) {\Huge $=$};

\pic[tqft/reverse pair of pants, name = f, at = {(15,2)}]; 
\pic[tqft/cap,anchor=outgoing boundary 1, at = ( f-incoming boundary 1)];
\pic[tqft/pair of pants, anchor=incoming boundary 1, name = g, at = (f-outgoing boundary 1)];
\pic[tqft/cup,anchor=incoming boundary 1, at = (g-outgoing boundary 2)];

\node[text width = 1.3cm] at (18.7,0) {\Huge $=$};

\pic[tqft/pair of pants, name = h, at = {(21,2)}];
\pic[tqft/cylinder,anchor = incoming boundary 1, at = (h-outgoing boundary 1)];
\pic[tqft/reverse pair of pants,anchor=incoming boundary 1, name = i, at = (h-outgoing boundary 2)];
\pic[tqft/cap, anchor = outgoing boundary 1, at = (h-incoming  boundary 1)];
\pic[tqft/cylinder, anchor = outgoing boundary 1, at = (i-incoming boundary 2)];
\pic[tqft/cup, anchor = incoming boundary 1, at = (i-outgoing boundary 1)];

\end{scope}
\end{tikzpicture}
\end{center}
so the snake identity holds.
\end{proof}

As a consequence of Lemmas \ref{lemma:nondegenerate} and \ref{lemma:frobnondegenerate}, we have the following result.

\begin{prop}\label{prop:alpha}
Suppose that $(X,\epsilon,\eta,\delta,\mu)$ is a Frobenius object in $\rel$, and let $\alpha = \epsilon \circ \mu$ and $\beta = \delta \circ \eta$. Then there is a bijection $\hat{\alpha}: X \to X$ such that $\alpha$ and $\beta$ are of the form specified in Lemma \ref{lemma:nondegenerate}.
\end{prop}

%====
\subsection{Relations involving $\hat{\alpha}$}

Let $(X,\epsilon,\eta,\delta,\mu)$ be a Frobenius object in $\rel$. The bijection $\hat{\alpha}$ provides a close relationship between the unit $\eta$ and the counit $\epsilon$, as well as between the multiplication $\mu$ and comultiplication $\delta$.

\begin{prop}\label{prop:alphaunit}
For $x \in X$,
\begin{enumerate}
    \item $x \in \eta$ if and only if $\hat{\alpha}(x) \in \epsilon$.
    \item $x \in \epsilon$ if and only if $\hat{\alpha}(x) \in \eta$.
\end{enumerate}
\end{prop}

\begin{proof}
Suppose $x \in \eta$. By Proposition \ref{prop:alpha}, we have that
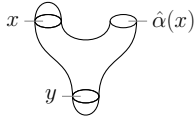
\begin{center}
\begin{tikzpicture}[scale=0.5, transform shape]
\begin{scope}[tqft/every boundary component/.style = {draw},tqft/cobordism edge/.style = {draw}]
\pic[tqft/cap, name = a, at = {(0,2)}];

\pic[tqft/reverse pair of pants,anchor=incoming boundary 1, name = b, at = (a-outgoing boundary 1)];

\pic[tqft/cup,anchor=incoming boundary 1, at=(b-outgoing boundary 1)]; 

\node[pin=180:\LARGE $y$] at (b-outgoing boundary 1) {};
\node[pin=180:\LARGE $x$] at (a-outgoing boundary 1) {};
\node[pin=0:\LARGE $\hat{\alpha}(x)$] at (b-incoming boundary 2) {};
\end{scope}
\end{tikzpicture}
\end{center}
is valid for some $y \in \epsilon$. But by Axiom (\textbf{U}), it must be that $\hat{\alpha}(x) = y \in \epsilon$. The other parts of the proof are similar.
\end{proof}

\begin{lemma}\label{lemma:comultiplication}
The following identities hold:
\begin{align}
(\idx \times \alpha) \circ (\delta \times \idx)  & = \mu \label{md1}\\
(\alpha\times \idx) \circ (\idx \times \delta) & = \mu \label{md2}
\end{align}
\end{lemma}

\begin{proof}
Using Axioms (\textbf{F}) and (\textbf{C}), we have 

\begin{center}
\begin{tikzpicture}[scale=0.5, transform shape]
\begin{scope}[tqft/every boundary component/.style = {draw},tqft/cobordism edge/.style = {draw}]

\pic[tqft/pair of pants, name =a, at = {(0,2)}];
\pic[tqft/cylinder, at = (a-outgoing boundary 1)];
\pic[tqft/reverse pair of pants,anchor=incoming boundary 1,name = b, at = (a-outgoing boundary 2)];
\pic[tqft/cylinder, anchor  = outgoing boundary 1, at = (b-incoming boundary 2)];
\pic[tqft/cup,anchor=incoming boundary 1, at = (b-outgoing boundary 1)];

\node[text width = 1.3cm] at (5,0) {\Huge $=$};

\pic[tqft/reverse pair of pants, name = c, at = {(6.3,2)}];
\pic[tqft/pair of pants, anchor=incoming boundary 1, name =d, at = (c-outgoing boundary 1)];
\pic[tqft/cup,anchor=incoming boundary 1, at = (d-outgoing boundary 2)];

\node[text width = 1.3cm] at (10.6,0) {\Huge $=$};

\pic[tqft/reverse pair of pants, at = {(12,1)}];
\end{scope}
\end{tikzpicture}
\end{center}
which proves \eqref{md1}. The proof of \eqref{md2} is similar.
\end{proof}

\begin{lemma} \label{md}
For $(x,y,z) \in X^3$, the following are equivalent:
\begin{enumerate}
    \item $(x,y,z) \in \mu$,
    \item $(y,\hat{\alpha}(x),z) \in \delta$,
    \item $(x,z,\hat{\alpha}^{-1}(y)) \in \delta$
\end{enumerate}
\end{lemma}
\begin{proof}
$(1 \iff 3)$ By \eqref{md1}, we have that $(x,y,z) \in \mu$ if and only if $(x,y,z) \in (\idx \times \alpha) \circ(\delta \times \idx)$. Graphically, 
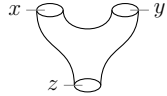
\begin{center}
\begin{tikzpicture}[scale=0.5, transform shape]
\begin{scope}[tqft/every boundary component/.style = {draw},tqft/cobordism edge/.style = {draw}]

\pic[tqft/reverse pair of pants, name =c, at = {(9.2,1)}];

\node[pin=180:\LARGE $x$] at (c-incoming boundary 1) {};
\node[pin=0:\LARGE $y$] at (c-incoming boundary 2) {};
\node[pin=180:\LARGE $z$] at (c-outgoing boundary 1) {};
\end{scope}
\end{tikzpicture}
\end{center}
is valid if and only if
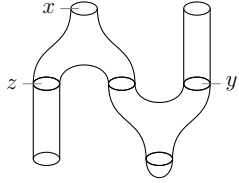
\begin{center}
\begin{tikzpicture}[scale=0.5, transform shape]
\begin{scope}[tqft/every boundary component/.style = {draw},tqft/cobordism edge/.style = {draw}]
\pic[tqft/pair of pants, name = a, at = {(0,2)}];

\pic[tqft/cylinder,anchor=incoming boundary 1, at = (a-outgoing boundary 1)];

\pic[tqft/reverse pair of pants,anchor=incoming boundary 1, name = b, at = (a-outgoing boundary 2)];

\pic[tqft/cylinder,anchor=outgoing boundary 1, at = (b-incoming boundary 2)];

\pic[tqft/cup,anchor=incoming boundary 1, at=(b-outgoing boundary 1)]; 

\node[pin=180:\LARGE $x$] at (a-incoming boundary 1) {};
\node[pin=180:\LARGE $z$] at (a-outgoing boundary 1) {};
\node[pin=0:\LARGE $y$] at (b-incoming boundary 2) {};
\end{scope}
\end{tikzpicture}
\end{center}
is valid. By Proposition \ref{prop:alpha}, this is the case if and only if
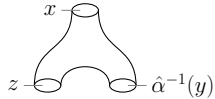
\begin{center}
\begin{tikzpicture}[scale=0.5, transform shape]
\begin{scope}[tqft/every boundary component/.style = {draw},tqft/cobordism edge/.style = {draw}]

\pic[tqft/pair of pants, name = a, at = {(0,2)}];

\node[pin=180:\LARGE $x$] at (a-incoming boundary 1) {};
\node[pin=180:\LARGE $z$] at (a-outgoing boundary 1) {};
\node[pin=0:\LARGE $\hat{\alpha}^{-1}(y)$] at (a-outgoing boundary 2) {};
\end{scope}
\end{tikzpicture}
\end{center}
is valid.
$(1 \iff 2)$ can be similarly proved using \eqref{md2}.
\end{proof}

\begin{prop}
\label{alpha} Let $(x,y,z) \in X^3$. Then $(x,y,z) \in \mu$ if and only if $(y, \hat{\alpha}(z),\hat{\alpha}(x))\in \mu$.
\end{prop}
\begin{proof}
By Lemma \ref{md}, $(x,y,z) \in \mu$ if and only if $(y,\hat{\alpha}(x),z) \in \delta$. Using $(1 \iff 3)$ from Lemma \ref{md}, we see that $(y,\hat{\alpha}(x),z) \in \delta$ if and only if $(y,\hat{\alpha}(z), \hat{\alpha}(x) \in \mu$.
\end{proof}

\begin{remark}\label{remark:rotation}
Graphically, the identity in Proposition \ref{alpha} asserts that
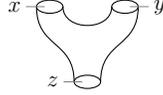
\begin{center}
\begin{tikzpicture}[scale=0.5, transform shape]
\begin{scope}[tqft/every boundary component/.style = {draw},tqft/cobordism edge/.style = {draw}]

\pic[tqft/reverse pair of pants, name = c, at = {(0,0)}];

\node[pin=180:\LARGE $x$] at (c-incoming boundary 1) {};
\node[pin=0:\LARGE $y$] at (c-incoming boundary 2) {};
\node[pin=180:\LARGE $z$] at (c-outgoing boundary 1) {};
\end{scope}
\end{tikzpicture}
\end{center}
is valid if and only if
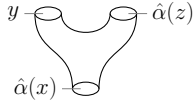
\begin{center}
\begin{tikzpicture}[scale=0.5, transform shape]
\begin{scope}[tqft/every boundary component/.style = {draw},tqft/cobordism edge/.style = {draw}]

\pic[tqft/reverse pair of pants, name =c, at = {(0,0)}];

\node[pin=180:\LARGE $y$] at (c-incoming boundary 1) {};
\node[pin=0:\LARGE $\hat{\alpha}(z)$] at (c-incoming boundary 2) {};
\node[pin=-180:\LARGE $\hat{\alpha}(x)$] at (c-outgoing boundary 1) {};
\end{scope}
\end{tikzpicture}
\end{center}
is valid. Informally, the bijection $\hat{\alpha}$ provides a way to rotate the boundaries of $\mu$ in the counterclockwise direction. 
\end{remark}

%===
\subsection{Associativity and coassociativity}
Let $X$ be a set.
\begin{definition}
A relation $\mu: X \times X \relto X$ is \emph{associative} if 
\[\mu \circ (\idx \times \mu) = \mu \circ (\mu \times \idx).\]
\end{definition}
\begin{definition}
A relation $\delta: X \relto X \times X$ is \emph{coassociative} if \[(\delta \times \idx) \circ \delta = (\idx \times \delta) \circ \delta.\]
\end{definition}

\begin{lemma}\label{lemma:associativity}
Suppose that $(X,\epsilon,\eta,\delta,\mu)$ is a Frobenius object in $\rel$. Then $\mu$ is associative and $\delta$ is coassociative.
\end{lemma}
\begin{proof}
Using Lemma \ref{lemma:comultiplication} and Axiom (\textbf{F}), we have
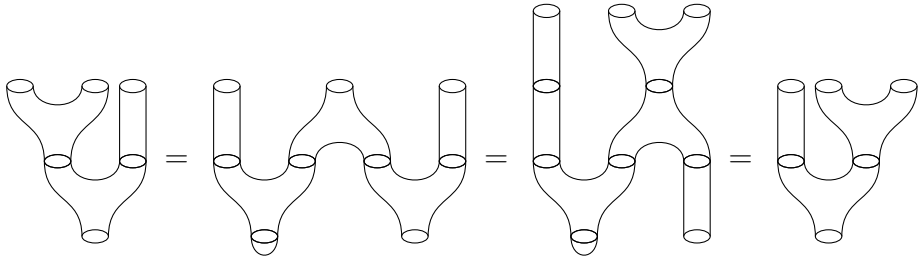
\begin{center}
\begin{tikzpicture}[scale=0.5, transform shape]
\begin{scope}[tqft/every boundary component/.style = {draw},tqft/cobordism edge/.style = {draw}]

\pic[tqft/reverse pair of pants, name = a, at = {(0,0)}];
\pic[tqft/reverse pair of pants,anchor=outgoing boundary 1, at= (a-incoming boundary 1)];
\pic[tqft/cylinder, anchor = outgoing boundary 1, at= (a-incoming  boundary 2)];

\node[text width= 1.3cm] at (3.5,0) {\Huge $=$};

\pic[tqft/reverse pair of pants, name = b, at = {(8.5,0)}];
\pic[tqft/cylinder, anchor = outgoing boundary 1, at = (b-incoming boundary 2)];
\pic[tqft/pair of pants,anchor=outgoing boundary 2, name =c, at = (b-incoming boundary 1)];
\pic[tqft/reverse pair of pants, anchor = incoming boundary 2, name = d, at = (c-outgoing boundary 1)];
\pic[tqft/cylinder, anchor = outgoing boundary 1, at = (d-incoming boundary 1)];
\pic[tqft/cup, anchor = incoming boundary 1, at = (d-outgoing boundary 1)];

\node[text width= 1.3cm] at (12,0) {\Huge $=$};

\pic[tqft/pair of pants, name = e, at = {(16,2)}];
\pic[tqft/cylinder, anchor = incoming boundary 1,  at=(e-outgoing boundary 2)];
\pic[tqft/reverse pair of pants, anchor = outgoing boundary 1 ,name = f, at = (e-incoming boundary 1)];
\pic[tqft/reverse pair of pants, anchor = incoming boundary 2, name = g, at = (e-outgoing boundary 1)]; 
\pic[tqft/cylinder, anchor = outgoing boundary 1, name = i, at=(g-incoming boundary 1)];
\pic[tqft/cylinder, anchor = outgoing boundary 1, at=(i-incoming boundary 1)];
\pic[tqft/cup, anchor = incoming boundary 1, at = (g-outgoing boundary 1)];

\node[text width= 1.3cm] at (18.5,0) {\Huge $=$};

\pic[tqft/reverse pair of pants, name = h, at = {(19.5,0)}];
\pic[tqft/cylinder, anchor = outgoing boundary 1, at=(h-incoming boundary 1)];
\pic[tqft/reverse pair of pants,anchor=outgoing boundary 1, at = (h-incoming boundary 2)];
\end{scope}
\end{tikzpicture}
\end{center}
This proves associativity of $\mu$. The proof of coassociativity of $\delta$ is similar.
\end{proof}

%====
\subsection{Source and target maps}\label{sub:st}

Let $(X,\epsilon,\eta,\delta,\mu)$ be a Frobenius object in $\rel$. Then the unit relation $\eta: \{\bullet\} \relto X$ is a subset $\eta \subseteq X$, the elements of which can be viewed as the ``units'' of $X$.
\begin{lemma}
\label{uniqueu}
\begin{enumerate}
    \item For every $x \in X$, there is a unique $s(x) \in \eta$ such that $(x,s(x),x) \in \mu$.
    \item For every $x \in X$, there is a unique $t(x) \in \eta$ such that $(t(x),x,x) \in \mu$.
\end{enumerate}
\end{lemma} 

\begin{proof}
Existence immediately follows from the Unit Axiom (\textbf{U}), so it remains to prove uniqueness.

Fix some $x \in X$, and assume that $u,u'\in \eta$ are such that $(x,u,x)$ and $(x,u',x)$ are in $\mu$. Then
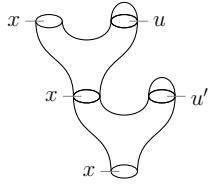
\begin{center}
\begin{tikzpicture}[scale=0.5, transform shape]
\begin{scope}[tqft/every boundary component/.style = {draw},tqft/cobordism edge/.style = {draw}]

\pic[tqft/reverse pair of pants, name = b, at = {(0,0 )}];
\pic[tqft/reverse pair of pants,anchor=outgoing boundary 1, name = c, at = (b-incoming boundary 1)];
\pic[tqft/cap,anchor=outgoing boundary 1,at = (b-incoming boundary 2)];
\pic[tqft/cap,anchor=outgoing boundary 1,at = (c-incoming boundary 2)];

\node[pin=180:\LARGE $x$] at (c-incoming boundary 1) {};
\node[pin=0:\LARGE $u$] at (c-incoming boundary 2) {};
\node[pin=180:\LARGE $x$] at (c-outgoing boundary 1) {};
\node[pin=180:\LARGE $x$] at (b-outgoing boundary 1) {};
\node[pin=0:\LARGE $u'$] at (b-incoming boundary 2) {};
\end{scope}
\end{tikzpicture}
\end{center}
is valid. By associativity, there exists some $z \in X$ such that
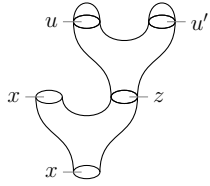
\begin{center}
\begin{tikzpicture}[scale=0.5, transform shape]
\begin{scope}[tqft/every boundary component/.style = {draw},tqft/cobordism edge/.style = {draw}]

\pic[tqft/reverse pair of pants, name = d, at = {(0,0 )}];
\pic[tqft/reverse pair of pants,anchor=outgoing boundary 1, name = e, at = (d-incoming boundary 2)];
\pic[tqft/cap,anchor=outgoing boundary 1, at = (e-incoming boundary 1)];
\pic[tqft/cap,anchor=outgoing boundary 1, at = (e-incoming boundary 2)];

\node[pin=180:\LARGE $u$] at (e-incoming boundary 1) {};
\node[pin=0:\LARGE $u'$] at (e-incoming boundary 2) {};
\node[pin=0:\LARGE $z$] at (e-outgoing boundary 1) {};
\node[pin=180:\LARGE $x$] at (d-outgoing boundary 1) {};
\node[pin=180:\LARGE $x$] at (d-incoming boundary 1) {};
\end{scope}
\end{tikzpicture}
\end{center}
is valid. But, by (\textbf{U}), it must be that $z = u = u'$. This proves uniqueness for the first statement; the proof for the second statement is similar.
\end{proof}
As a result of Lemma \ref{uniqueu}, we obtain maps $s,t: X \to \eta$ that we will call the \emph{source} and \emph{target} maps. The following Propositions describe how the source and target maps are compatible with $\hat{\alpha}$ and $\mu$.

\begin{prop}\label{prop:st-alpha}
$s = t \circ \hat{\alpha}$.
\end{prop}
\begin{proof}
This is an immediate consequence of Proposition \ref{alpha}.
\end{proof}

\begin{prop}\label{prop:st-mu}
Let $(x,y,z) \in \mu$. Then
\begin{enumerate}
    \item $s(x) = t(y)$,
    \item $s(y) = s(z)$, and
    \item $t(x) = t(z)$.
\end{enumerate}
\end{prop}
\begin{proof}
Suppose that $(x,y,z) \in \mu$. Then, by Lemma \ref{uniqueu}, we have that 
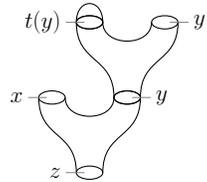
\begin{center}
\begin{tikzpicture}[scale=0.5, transform shape]
\begin{scope}[tqft/every boundary component/.style = {draw},tqft/cobordism edge/.style = {draw}]

\pic[tqft/reverse pair of pants, name = d, at = {(0,0)}];
\pic[tqft/reverse pair of pants,anchor=outgoing boundary 1, name = e, at = (d-incoming boundary 2)]; 

\pic[tqft/cap,anchor=outgoing boundary 1, at = (e-incoming boundary 1)];

\node[pin=180:\LARGE $t(y)$] at (e-incoming boundary 1) {};
\node[pin=0:\LARGE $y$] at (e-incoming boundary 2) {};
\node[pin=0:\LARGE $y$] at (e-outgoing boundary 1) {};
\node[pin=180:\LARGE $z$] at (d-outgoing boundary 1) {};
\node[pin=180:\LARGE $x$] at (d-incoming boundary 1) {};
\end{scope}
\end{tikzpicture}
\end{center}
is valid. By associativity, there then exists some $w \in X$ such that
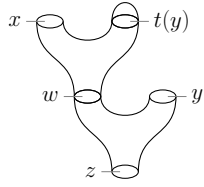
\begin{center}
\begin{tikzpicture}[scale=0.5, transform shape]
\begin{scope}[tqft/every boundary component/.style = {draw},tqft/cobordism edge/.style = {draw}]

\pic[tqft/reverse pair of pants, name =b, at = {(0,0)}];

\pic[tqft/reverse pair of pants,anchor=outgoing boundary 1, name =c, at = (b-incoming boundary 1)];

\pic[tqft/cap,anchor=outgoing boundary 1, at = (c-incoming boundary 2)];

\node[pin=180:\LARGE $x$] at (c-incoming boundary 1) {};
\node[pin=0:\LARGE $t(y)$] at (c-incoming boundary 2) {};
\node[pin=180:\LARGE $w$] at (c-outgoing boundary 1) {};
\node[pin=180:\LARGE $z$] at (b-outgoing boundary 1) {};
\node[pin=0:\LARGE $y$] at (b-incoming boundary 2) {};
\end{scope}
\end{tikzpicture}
\end{center}
is valid. By (\textbf{U}), it must be that $w=x$. It then follows from Lemma \ref{uniqueu} that $s(x) = t(y)$. This proves (1).

Using Proposition \ref{alpha}, we have that $(y,\hat{\alpha}(z),\hat{\alpha}(x))$ and $(\hat{\alpha}^{-1}(z),x, \hat{\alpha}^{-1}(y))$ are in $\mu$. By (1), it follows that $s(y) = t(\hat{\alpha}(z))$ and $s(\hat{\alpha}^{-1}(z)) = t(x)$. Using Proposition \ref{prop:st-alpha}, we deduce (2) and (3).
\end{proof}

\subsection{Dual structures}
The symmetry of axioms (\textbf{U}), (\textbf{C}), and (\textbf{F}) imply that, given any Frobenius object $X = (X, \epsilon, \eta, \delta, \mu)$ in $\rel$, we can obtain various ``dual'' Frobenius structures on the same underlying set $X$.

The \emph{rotation} of $X$, denoted $X^\rot = (X, \epsilon^\rot, \eta^\rot, \delta^\rot, \mu^\rot)$, has structure morphisms that are rotated by an angle of $\pi$ with respect to those of $X$. Specifically, using the identification $\{\bullet\} \times X \cong X \times \{\bullet\} \cong X$, we have
\begin{itemize}
    \item $\epsilon^\rot = \eta$,
    \item $\eta^\rot = \epsilon$,
    \item $(x,y,z) \in \delta^\rot$ if and only if $(z,y,x) \in \mu$,
    \item $(x,y,z) \in \mu^\rot$ if and only if $(z,y,x) \in \delta$.
\end{itemize}

The \emph{dagger} of $X$, denoted $X^\dagger = (X,\epsilon^\dagger, \eta^\dagger, \delta^\dagger, \mu^\dagger)$, has structure morphisms that are vertically reflected with respect to those of $X$. Specifically,
\begin{itemize}
    \item $\epsilon^\dagger = \eta$,
    \item $\eta^\dagger = \epsilon$,
    \item $(x,y,z) \in \delta^\dagger$ if and only if $(y,z,x) \in \mu$,
    \item $(x,y,z) \in \mu^\dagger$ if and only if $(z,x,y) \in \delta$.
\end{itemize}

The \emph{opposite} of $X$, denoted $X^\op = (X,\epsilon, \eta, \delta^\op, \mu^\op)$, has structure morphisms that are horizontally reflected with respect to those of $X$. Specifically, $X^\op$ has the same counit and unit morphisms as $X$, and
\begin{itemize}
    \item $(x,y,z) \in \delta^\op$ if and only if $(x,z,y) \in \delta$,
    \item $(x,y,z) \in \mu^\op$ if and only if $(y,x,z) \in \mu$.
\end{itemize}

\begin{remark}\label{rmk:nakayama} 
From Proposition \ref{prop:alphaunit} and Lemma \ref{md}, one can see that $\hat{\alpha}$ is an isomorphism from $X$ to $X^\rot$. Since $\hat{\alpha}^\rot = \hat{\alpha}$, it follows that $\hat{\alpha}^2$ is an automorphism of $X$. This is the analogue of the Nakayama automorphism of a Frobenius algebra \cite{nakayama2}\footnote{See \cite{fuchs} for a definition of Nakayama automorphism of a Frobenius object in any sovereign monoidal category.}. An example of a Frobenius object in $\rel$ with nontrivial Nakayama automorphism is given in Example \ref{example:torus}.
\end{remark}

%========================
\section{Simplicial Sets}\label{sec:simplicial}

The identities proven in Section \ref{sub:st} are actually part of a much richer structure, that of a simplicial set. We briefly review the definitions here.

\begin{definition} \label{dfn:simplicial}A \emph{simplicial set} $\mathcal{X}$ is a sequence $X_0, X_1, \dots$  of sets equipped with maps $d_i^q: X_q \to X_{q-1}$ (called \emph{face maps}), $0 \leq i \leq q$, and $s_i^q : X_q \to X_{q+1}$ (called \emph{degeneracy maps}), $0 \leq i \leq q$,  such that
\begin{align}
d_i^{q-1}d_j^q &= d_{j-1}^{q-1}d_i^q, \qquad i < j,\label{eqn:twoface}\\
s_i^{q+1}s_j^q &= s_{j+1}^{q+1}s_i^q, \qquad i \leq j, \label{eqn:twodegen}\\
d_i^{q+1}s_j^q &= \begin{cases}
s_{j-1}^{q-1}d_i^q, & i< j,\\
\id, & i = j \mbox{ or }j+1, \\
s_j^{q-1}d_{i-1}^q, & i > j+1. \end{cases}\label{eqn:facedegen}
\end{align}
\end{definition}

We will also need to consider $n$-truncated versions of simplicial sets, which only include data going up to $X_n$:
\begin{definition}\label{dfn:nsimplicial}
An \emph{$n$-truncated simplicial set} $\mathcal{X}$ is a sequence $X_0, X_1, \dots X_n$ of sets equipped with face maps $d_i^q: X_q \to X_{q-1}$, $0 \leq i \leq q \leq n$, and degeneracy map $s_i^q : X_q \to X_{q+1}$, $0 \leq i \leq q < n$, satisfying \eqref{eqn:twoface}--\eqref{eqn:facedegen} whenever the maps on both sides of the equation exist.
\end{definition}

Suppose that $\mathcal{X}$ is a (possibly $n$-truncated) simplicial set. For $1 \leq q \leq n+1$, we let $\Delta_q \mathcal{X}$ denote the set of $(q+1)$-tuples $(\zeta_0, \dots \zeta_q)$, $\zeta_i \in X_{q-1}$, such that
\begin{equation}\label{eqn:horncompat}
     d_i^{q-1} \zeta_j = d_{j-1}^{q-1} \zeta_i
\end{equation}
for $i < j$. Intuitively, an element of $\Delta_q \mathcal{X}$ can be viewed a $(q-1)$-dimensional outline of a $q$-simplex. Given such an outline, there may or may not be a $q$-simplex whose boundary is the given outline. Furthermore, even if such a $q$-simplex exists, it may not be unique. To address these issues, we consider the \emph{boundary map} $\delta_q: X_q \to \Delta_q \mathcal{X}$,
\[ \delta_q(w) = (d_0^q w, \dots, d_q^q w).\]
\begin{definition}
A simplicial set $\mathcal{X}$ is called \emph{$n$-coskeletal} if $\delta_q$ is a bijection for $q>n$.
\end{definition}

It is well-known (see, for example, \cite{artin-mazur}) that any $n$-truncated simplicial set has a unique extension to an $n$-coskeletal simplicial set. The extension can be recursively constructed by taking $X_{n+1} = \Delta_{n+1} \mathcal{X}$.

\section{Frobenius objects in $\rel$ and simplicial sets}

In this section, we present the main results of the paper. Given a Frobenius object in $\rel$, we associate a simplicial set. We then find a characterization of the simplicial sets that arise in this way, thus obtaining a one-to-one correspondence.

\subsection{From Frobenius objects in $\rel$ to simplicial sets} \label{sub:frobtosimp}

Let (X, $\mu$, $\delta$, $\epsilon$, $\eta$) be a Frobenius object in $\rel$. We can construct a $2$-truncated simplicial set as follows. Set
\begin{align*}
X_0 &: = \eta \subseteq X,\\
X_1 & := X,\\
X_2 & := \mu \subseteq X^3.
\end{align*}
The degeneracy map $s_0^0: X_0 \to X_1$ is the inclusion map. The face maps $d_i^1: X_1 \to X_0$ are the source and target maps: 
\begin{align*}
     d_0^1 &= s, & d_1^1 &= t.
\end{align*}
The degeneracy maps $s_i^1: X_1 \to X_2$ are given by
\begin{align*}
    s_0^1(x) &= (t(x),x,x), & s_1^1(x) &= (x,s(x),x).
\end{align*}
The face maps $d_i^2: X_2 \to X_1$ are the projection maps:
\begin{align*}
    d_0^2 (x,y,z) &= y, & d_1^2(x,y,z) &= z, & d_2^2(x,y,z) &= x.
\end{align*}

\begin{thm}\label{thm:trunc}
The sets $X_0,X_1,X_2$, with face and degeneracy maps defined as above, form a $2$-truncated simplicial set. Consequently, there exists a unique $2$-coskeletal extension $\mathcal{X}$.
\end{thm}

We leave the verification of Theorem \ref{thm:trunc} as a straightforward exercise for the reader. But we remark that the nontrivial cases of \eqref{eqn:twoface} follow from Proposition \ref{prop:st-mu}.

Our next goal will be to identify the special properties possessed by the simplicial set $\mathcal{X}$ that arises via Theorem \ref{thm:trunc}. In what follows, we will use a modified version of the $2$-boundary map $\delta_2$, given by $\delta_2(\zeta) = (d_2^2 \zeta, d_0^2 \zeta, d_1^2 \zeta)$ for $\zeta \in X_2$; this choice of convention allows us to identify $X_2$ with its image under $\delta_2$ without permuting the components.

The following statements follow by construction and from Axiom (\textbf{U}). 
\begin{prop}\label{prop:x2relations}
\begin{enumerate}
    \item The $2$-boundary map $\delta_2: X_2 \to \Delta_2 \mathcal{X}$ is injective.
    \item For any $\zeta \in X_2$, if $d_0^2\zeta = s_0^0 u$ for some $u \in X_0$, then $d_2^2 \zeta = d_1^2 \zeta$. Similarly, if $d_2^2\zeta = s_0^0 u$ for some $u \in X_0$, then $d_0^2 \zeta = d_1^2 \zeta$.
\end{enumerate}
\end{prop}

The simplicial set $\mathcal{X}$ satisfies the following lifting property involving $3$-simplices. By construction, $X_3$ consists of $4$-tuples $(\zeta_0, \zeta_1, \zeta_2, \zeta_3)$, where $\zeta_i \in X_2$, satisfying \eqref{eqn:horncompat}. If, identifying $X_2$ with its image under $\delta_2$, we write $\zeta_i = (x_{i,2},x_{i,0},x_{i,1})$, then \eqref{eqn:horncompat} requires that $x_{j,i} = x_{i,j-1}$ for $i \leq j$. These equations can be represented by a fully-labeled associativity diagram:

\begin{equation}\label{eqn:x3}
\begin{tikzpicture}[scale=0.5, transform shape, baseline=(current  bounding  box.center)]
\begin{scope}[tqft/every boundary component/.style = {draw},tqft/cobordism edge/.style = {draw}]

\pic[tqft/reverse pair of pants, name = a, at = {(0,0)}];
\pic[tqft/reverse pair of pants, name= b, anchor=outgoing boundary 1, at= (a-incoming boundary 1)];
\node[pin=180:\LARGE $x_{3,2}$] at (b-incoming boundary 1) {};
\node[pin=0:\LARGE $x_{3,0}$] at (b-incoming boundary 2) {};
\node[pin=180:{\LARGE $x_{3,1}=x_{1,2}$}] at (a-incoming boundary 1) {};
\node[pin=0:\LARGE $x_{1,0}$] at (a-incoming boundary 2) {};
\node[pin=180:\LARGE $x_{1,1}$] at (a-outgoing boundary 1) {};

\node[text width= 1.3cm] at (5.3,0) {\Huge $=$};

\pic[tqft/reverse pair of pants, name = h, at = {(8,0)}];
\pic[tqft/reverse pair of pants, name=j, anchor=outgoing boundary 1, at = (h-incoming boundary 2)];

\node[pin=180:\LARGE $x_{0,2}$] at (j-incoming boundary 1) {};
\node[pin=0:\LARGE $x_{0,0}$] at (j-incoming boundary 2) {};
\node[pin=180:\LARGE $x_{2,2}$ ] at (h-incoming boundary 1) {};
\node[pin=0:{\LARGE $x_{0,1}=x_{2,0}$}] at (h-incoming boundary 2) {};
\node[pin=180:\LARGE $x_{2,1}$] at (h-outgoing boundary 1) {};
\end{scope}
\end{tikzpicture}
\end{equation}
The equal sign in \eqref{eqn:x3} means that the corresponding boundary labels are equal; so $x_{3,2} = x_{2,2}$, $x_{3,0} = x_{0,2}$, etc. 
Thus, such fully-labeled associativity diagrams provide a way to visualize elements of $X_3$.

The following statement is a direct consequence of the associativity property (Lemma \ref{lemma:associativity}). For our purposes, the important fact is that it provides a way of characterizing associativity in terms of the simplicial structure.

\begin{prop} \label{prop:simplicialassociativity}
\begin{enumerate}
    \item Let $\zeta_0, \zeta_2 \in X_2$ be such that $d_0^2 \zeta_2 = d_1^2 \zeta_0$. Then there exists some $\gamma \in X_3$ such that $d_0^3 \gamma = \zeta_0$ and $d_2^3 \gamma = \zeta_2$.
    \item Let $\zeta_1, \zeta_3 \in X_2$ be such that $d_1^2 \zeta_3 = d_2^2 \zeta_1$. Then there exists some $\gamma \in X_3$ such that $d_1^3 \gamma = \zeta_1$ and $d_3^3 \gamma = \zeta_3$.
\end{enumerate}
\end{prop}

Recall from Proposition \ref{prop:alpha} that, if $X$ is a Frobenius object in $\rel$, then there is an associated automorphism $\hat{\alpha}$ of $X = X_1$. We now consider the compatibility of $\mathcal{X}$ with $\hat{\alpha}$.

As a result of Proposition \ref{prop:x2relations}, we can identify $X_2$ with a subset of $\Delta_2 \mathcal{X}$, giving us a filtration of sets $X_2 \subseteq \Delta_2 \mathcal{X} \subseteq (X_1)^3$. Recall from Remark \ref{remark:rotation} the ``rotation'' action of $\hat{\alpha}$ on $(X_1)^3$, given by $(x,y,z) \mapsto (y, \hat{\alpha}(z), \hat{\alpha}(x))$. The following is immediate from Proposition \ref{alpha}.
\begin{prop}\label{prop:alphafilter}
$X_2$ is invariant under the $\Z$-action generated by the rotation action of $\hat{\alpha}$.
\end{prop}

\begin{remark}
It can also be shown, using Propositions \ref{prop:alphaunit} and \ref{prop:st-alpha}, that the rotation action of $\hat{\alpha}$ preserves $\Delta_2 \mathcal{X}$. Thus the filtration $X_2 \subseteq \Delta_2 \mathcal{X} \subseteq (X_1)^3$ is invariant.
\end{remark}

\subsection{From simplicial sets to Frobenius objects in $\rel$}

Let $\mathcal{X}$ be a $2$-coskeletal simplicial set, equipped with an automorphism $\hat{\alpha}$ of $X_1$, satisfying the properties in Propositions \ref{prop:x2relations}, \ref{prop:simplicialassociativity}, and \ref{prop:alphafilter}. From this, we can construct a Frobenius object in $\rel$, defined as follows. Set $X = X_1$, with the structure relations given by
\begin{itemize}
\item $\mu = X_2 \subseteq X^3$,
\item $\eta = s_0^0(X_0) \subseteq X$,
\item $\epsilon = (\hat{\alpha} \circ s_0^0)(X_0) \subseteq X$,
\item $\delta = \{(y, \hat{\alpha}(x), z) \in X^3 \suchthat (x,y,z) \in \mu\} \subseteq X^3$.
\end{itemize}

\begin{thm}\label{thm:simptofrob}
$(X, \epsilon, \eta, \delta, \mu)$ is a Frobenius object in $\rel$.
\end{thm}

\begin{proof}
Axiom (\textbf{U}) follows directly from \eqref{eqn:facedegen} and property (2) of Proposition \ref{prop:x2relations}.

From the definitions of $\epsilon$ and $\delta$, we have
\begin{align*}
    (\epsilon \times \idx) \circ \delta &= \{ (x,z) \suchthat (x,w,z) \in \delta \mbox{ for some } w \in \epsilon\} \\
    &= \{ (x,z) \suchthat (\hat{\alpha}^{-1}(w), x,z) \in \mu \mbox{ for some } w \in \epsilon\} \\
    &= \{ (x,z) \suchthat (u,x,z) \in \mu \mbox{ for some } u \in \eta\},
\end{align*}
which by Axiom (\textbf{U}) equals $\{(x,z) \suchthat x=z\} = \idx$. This proves part of Axiom (\textbf{C}). Similarly, also using the property of Proposition \ref{prop:alphafilter}, we have
\begin{align*}
    (\idx \times \epsilon) \circ \delta &= \{ (x,z) \suchthat (x,z,w) \in \delta \mbox{ for some } w \in \epsilon\} \\
    &= \{ (x,z) \suchthat (\hat{\alpha}^{-1}(z),x,w) \in \mu \mbox{ for some } w \in \epsilon\} \\
    &= \{ (x,z) \suchthat (\hat{\alpha}^{-1}(w), \hat{\alpha}^{-1}(z),\hat{\alpha}^{-1}(x)) \in \mu \mbox{ for some } w \in \epsilon\} \\
    &= \{ (x,z) \suchthat (u,\hat{\alpha}^{-1}(z), \hat{\alpha}^{-1}(x)) \in \mu \mbox{ for some } u \in \eta\} \\
    &= \{ (x,z) \suchthat \hat{\alpha}^{-1}(z) = \hat{\alpha}^{-1}(x) \}\\
    &= \idx.
\end{align*}
This proves the other part of Axiom (\textbf{C}).

The relation $(1 \times \mu) \circ (\delta \times 1)$ consists of $(x_1,x_2,x_3,x_4)$ such that $(x_1,x_3,w) \in \delta$ and  $(w, x_2, x_4) \in \mu$ for some $w \in  X$. By definition of $\delta$, the condition $(x_1,x_3,w) \in \delta$ is equivalent to $(\hat{\alpha}^{-1}(x_3),x_1,w) \in \mu$. Thus we have one side of an associativity diagram:
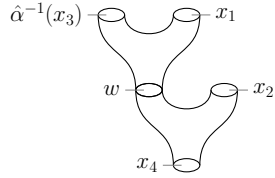
\begin{center}
\begin{tikzpicture}[scale = 0.5,transform shape]

\begin{scope}[tqft/every boundary component/.style = {draw},tqft/cobordism edge/.style = {draw}]

\pic[tqft/reverse pair of pants, name = a, at = {(0,0)}];
\pic[tqft/reverse pair of pants,anchor=outgoing boundary 1,name = b, at= (a-incoming boundary 1)];

\node[pin=180: \LARGE $w$] at(a-incoming boundary 1) {};
\node [pin=0: \LARGE $x_2$] at (a-incoming boundary 2) {};
\node [pin=180: \LARGE $x_4$] at (a-outgoing boundary 1) {};
\node [pin = 180: \LARGE $\hat{\alpha}^{-1}(x_3)$] at (b-incoming boundary 1) {};
\node [pin=0:\LARGE $x_1$]at (b-incoming boundary 2)  {};
\end{scope}
\end{tikzpicture}
\end{center}
By property (2) of Proposition \ref{prop:simplicialassociativity}, there exists some $w' \in X$ such that $(x_1, x_2, w') \in \mu$ and $(\hat{\alpha}^{-1}(x_3), w', x_4) \in \mu$. The latter is equivalent to $(w', x_3, x_4) \in \delta$. Together, $(x_1, x_2, w') \in \mu$ and $(w', x_3, x_4) \in \delta$ for some $w' \in X$ are equivalent to $(x_1, x_2, x_3, x_4) \in \delta \circ \mu$. This proves one part of Axiom (\textbf{F}). The proof of the other part is similar.
\end{proof}

\section{Groupoids and Frobenius objects in $\rel$}\label{sec:groupoids}

In \cite{hcc}, it was shown that special dagger Frobenius objects in $\rel$ correspond to groupoids. In this section, we describe how their result fits as a special case of the correspondence given by Theorems \ref{thm:trunc} and \ref{thm:simptofrob}.

Recall that a \emph{groupoid} is a small category where all arrows are invertible. In more concrete terms, a groupoid $G_1 \arrows G_0$ consists of sets $G_0, G_1$, equipped with \emph{source} and \emph{target} maps $s,t: G_1 \to G_0$ and a \emph{multiplication} operation $(g,h) \mapsto g \cdot h$, defined when $s(g) = t(h)$ for $g,h \in G_1$, such that
\begin{enumerate}
    \item $s(g \cdot h) = s(h)$ and $t(g \cdot h) = t(g)$ for $g,h \in G_1$ such that $s(g) = t(h)$,
    \item $(g\cdot h) \cdot k = g \cdot (h \cdot k)$ for all $g,h,k \in G_1$ such that $s(g) = t(h)$ and $s(h) = t(k)$,
    \item there exists an \emph{identity} map $e: G_0 \to G_1$ such that $s \circ e = t \circ e = \id$ and $g \cdot e(s(g)) = e(t(g)) \cdot g = g$ for all $g \in G_1$,
    \item there exists an \emph{inverse map} $G_1 \to G_1$, $g \mapsto g^{-1}$, such that $s(g^{-1}) = t(g)$, $t(g^{-1}) = s(g)$, $g \cdot g^{-1} = e(t(g))$, and $g^{-1} \cdot g = e(s(g))$ for all $g \in G_1$.
\end{enumerate}
As with groups, the identity and inverse maps are unique.

Given a groupoid $G_1 \arrows G_0$, there is a standard procedure for constructing a simplicial set, called the \emph{nerve} (see, e.g. \cite{segal}). The nerve of $G_1 \arrows G_0$ is the $2$-coskeletal simplicial set $\mathcal{X}$, where $X_0 = G_0$, $X_1 = G_1$, 
\[ X_2 = \{(g,h) \in G_1 \times G_1 \suchthat s(g) = t(h) \},\]
and where the face and degeneracy maps are as follows:
\begin{align*}
    s_0^0 &= e, & d_0^1 &= s, & d_1^1 &= t,\\
    d_0^2(g,h) &= h, & d_1^2(g,h) &= g \cdot h, & d_2^2(g,h) &= g,\\
    s_0^1(g) &= (e(t(g)),g), & s_1^1(g) &= (g, e(s(g))). &&
\end{align*}
From \eqref{eqn:x3}, it follows that the elements of $X_3$ can be identified with triplets $(g,h,k)$ such that $s(g) = t(h)$ and $s(h) = t(k)$, with 
\begin{align*}
    d_0^3(g,h,k) &= (h,k), & d_1^3(g,h,k) &= (g\cdot h, k),\\
    d_2^3(g,h,k) &= (g,h \cdot k), & d_3^3(g,h,k) &= (g,h).
\end{align*}

Note that this simplicial set comes naturally equipped with an automorphism of $X_1 = G_1$, given by the inverse map.

The following result shows that the nerve of a groupoid satisfies the conditions of Theorem \ref{thm:simptofrob} and therefore corresponds to a Frobenius object in $\rel$.
\begin{prop}
Let $\mathcal{X}$ be the nerve of a groupoid $G_1 \arrows G_0$. Then $\mathcal{X}$, with $\hat{\alpha}(g) = g^{-1}$, satisfies the properties in Propositions \ref{prop:x2relations}, \ref{prop:simplicialassociativity}, and \ref{prop:alphafilter}.
\end{prop}
\begin{proof}
Part (1) of Proposition \ref{prop:x2relations} is immediate. For part (2), we observe that, for $\zeta = (g,h) \in X_2$, the condition $d_0^2 \zeta = s_0^0 u$ implies that $h = e(s(g))$, and therefore $d_2^2 \zeta = d_1^2 \zeta = g$. Similarly, $d_2^2 \zeta = s_0^0 u$ implies that $g = e(t(h))$, and therefore $d_0^2 \zeta = d_1^2 \zeta = h$.

The property in Proposition \ref{prop:simplicialassociativity} is a straightforward consequence of the associativity property of the groupoid.

To understand the rotation action of $\hat{\alpha}$ on $X_2$, we observe that the image of $(g,h) \in X_2$ in $(X_1)^3$ is $(g, h, g \cdot h)$. Since $\hat{\alpha}$ is the inverse map, the rotation action gives $(h, (g \cdot h)^{-1}, g^{-1})$. This is also in the image of $X_2$, since $h \cdot(g \cdot h)^{-1} = g^{-1}$.
\end{proof}

In \cite{hcc}, Heunen, Contreras, and Cattaneo constructed a Frobenius object in $\rel$ associated to a groupoid $G_1 \arrows G_0$ by setting $X = G_1$, with structure relations given by
\begin{itemize}
    \item $\mu = \{(g,h,g \cdot h) \suchthat s(g) = t(h)\}$,
    \item $\delta = \{(g \cdot h, g, h) \suchthat s(g) = t(h)\}$,
    \item $\eta = \epsilon = e(G_0)$.
\end{itemize}
We observe that their construction coincides with that of Theorem \ref{thm:simptofrob} in the case where $\mathcal{X}$ is the nerve of $G_1 \arrows G_0$. In particular, the definitions of $\mu$ and $\eta$ are identical, the agreement in $\epsilon$ is due to the fact that $e(G_0)$ is invariant under the inverse map, and the agreement in $\delta$ holds because the sets $\{(g \cdot h, g, h) \suchthat s(g) = t(h)\}$ and $\{(h, g^{-1}, g \cdot h) \suchthat s(g) = t(h)\}$ are equal. 

It is immediate from the construction that the Frobenius objects that arise from groupoids satisfy the \emph{dagger} property $X^\dagger = X$ and the \emph{special} property $\mu \circ \delta = \idx$. The converse is proven in \cite{hcc}.

\section{Commutative Frobenius objects in $\rel$}\label{sec:commutative}

Two-dimensional TQFTs (with values in a symmetric monoidal category $\cat$) are classified by commutative Frobenius objects in $\cat$. In this section, we consider commutative Frobenius objects in $\rel$ and find additional properties implied by commutativity on the corresponding simplicial set.

Let $(X, \epsilon, \eta, \delta, \mu)$ be a Frobenius object in $\rel$. Let $\tau: X \times X \relto X \times X$ denote the twist relation $\tau = \{(x,y,y,x) \suchthat x,y \in X\}$.

\begin{definition}
$X$ is \emph{commutative} if $\mu \circ \tau = \mu$. 
\end{definition}

\begin{prop}
$X$ is commutative if and only if $X = X^\op$.
\end{prop}
\begin{proof}
The result is a consequence of the fact that commutativity implies the cocommutativity property $\tau \circ \delta = \delta$. To prove this, we use Lemma \ref{md} and commutativity to check that 
\begin{align*} (x,y,z) \in \tau \circ \delta \iff& (x,z,y) \in \delta \\
\iff& (\hat{\alpha}(y),\hat{\alpha}(z), \hat{\alpha}(x)) \in \mu \\
\iff& (\hat{\alpha}(z),\hat{\alpha}(y), \hat{\alpha}(x)) \in \mu \\
\iff& (x,y,z) \in \delta. \qedhere
\end{align*}
\end{proof}

There is also a notion of opposite for simplicial sets; if $\mathcal{X}$ is a simplicial set, then the opposite $\mathcal{X}^\op$ has the same underlying sets $X_\bullet$, with the indices for face and degeneracy maps reversed:
\begin{align*}
    (d^\op)_i^q &= d_{q-i}^q, & (s^\op)_i^q &= s_{q-i}^q.
\end{align*}
It is straightforward to see from the construction in Section \ref{sub:frobtosimp} that, if $X$ is a Frobenius object with corresponding simplicial set $\mathcal{X}$, then $\mathcal{X}^\op$ is the simplicial set that corresponds to $X^\op$. Thus, we have the following:
\begin{prop}\label{prop:commop}
Let $X$ be a commutative Frobenius object in $\rel$, with corresponding simplicial set $\mathcal{X}$. Then $\mathcal{X} = \mathcal{X}^\op$.
\end{prop}
A consequence of Proposition \ref{prop:commop} is that, if $\mathcal{X}$ is the simplicial set corresponding to a commutative Frobenius object, then $d_1^1 = d_0^1$, and therefore $\mathcal{X}$ decomposes as a disjoint union of simplicial sets with only one $0$-simplex. Proposition \ref{prop:st-alpha} implies that this decomposition is preserved by $\hat{\alpha}$, so commutative Frobenius objects always decompose as a disjoint union of those with one-element unit sets. In particular, the only groupoids that correspond to commutative Frobenius objects are disjoint unions of abelian groups.

The following result gives a property that commutative Frobenius objects and groupoids have in common.

\begin{prop}
Let $X$ be a commutative Frobenius object in $\rel$. Then $\hat{\alpha}$ is an involution. 
\end{prop}
\begin{proof}
By Proposition \ref{prop:alpha}, $\alpha = \{(x, \hat{\alpha}(x)) \suchthat x \in X\}$. By commutativity, it follows that $(\hat{\alpha}(x), x) \in \alpha$ for all $x \in X$, and therefore that $x = \hat{\alpha}^2(x)$.
\end{proof}

\section{Exterior algebras and cohomology rings}\label{sec:cohomology}

Recall that a standard example of a Frobenius algebra is the de Rham cohomology of a compact oriented manifold. There doesn't seem to be any direct way to convert this construction into a Frobenius object in $\rel$. Nonetheless, as we see in this section, there is an analogous construction, provided the manifold is equipped with a Riemannian metric.

\subsection{Exterior algebras}\label{sub:ext}

Let $V$ be an $n$-dimensional vector space over $\R$, equipped with an inner product and a choice of orientation. From this data, we obtain an associated volume form $\nu \in \bigwedge^n V$, an extension of the inner product to $\bigwedge V$, and the Hodge star operator $\star: \bigwedge^k V \to \bigwedge^{n-k} V$, related by the equation
\[ \lambda \wedge \star \theta = \langle \lambda, \theta \rangle \nu \]
for $\lambda, \theta \in \bigwedge^k V$.

We define a $2$-truncated simplicial set as follows. Set $X_0 = \{1\}$,
\[ X_1 = \{ \lambda \in \bigwedge V \suchthat \|\lambda\| = 1\},\]
\[ X_2 = \{ (\lambda, \theta, \phi) \in (X_1)^3 \suchthat \lambda \wedge \theta \wedge \star \phi = \nu.\}\]
We always implicitly assume that forms are of the appropriate degree for equations to make sense; for example, in order for $(\lambda, \theta, \phi)$ to be in $X_2$ with $\lambda \in \bigwedge^k V$, $\theta \in \bigwedge^\ell V$, and $\phi \in \bigwedge^m V$, it is necessary that $m = k+\ell$.

The face and degeneracy maps between $X_0$ and $X_1$ are obvious. The maps between $X_1$ and $X_2$ are
\begin{align*}
  d_0^2(\lambda, \theta, \phi) &= \theta, & d_1^2(\lambda, \theta, \phi) &= \phi, & d_2^2(\lambda, \theta, \phi) &= \lambda,
\end{align*}
\begin{align*}
  s_0^1(\lambda) &= (1, \lambda, \lambda), &   s_1^1(\lambda) &= (\lambda, 1, \lambda).
\end{align*}
The identities \eqref{eqn:twoface}--\eqref{eqn:facedegen} are immediate. Let $\mathcal{X}$ denote the $2$-coskeletal extension.

\begin{prop}\label{prop:exterior}
The simplicial set $\mathcal{X}$, with $\hat{\alpha} = \star$, satisfies the properties in Propositions \ref{prop:x2relations}, \ref{prop:simplicialassociativity}, and \ref{prop:alphafilter}, and therefore corresponds to a Frobenius object in $\rel$.
\end{prop}

The proof of the proposition will make use of the following lemma.
\begin{lemma}\label{lemma:ext}
Let $(\lambda, \theta, \phi) \in X_2$. Then $\| \lambda \wedge \theta\| = 1$ and $\phi = \lambda \wedge \theta$.
\end{lemma}
\begin{proof}
From the definition of $X_2$, we have $\langle \lambda \wedge \theta, \phi \rangle = 1$. The Cauchy-Schwarz inequality implies that $1 \leq \|\lambda \wedge \theta \| \|\phi\| = \|\lambda \wedge \theta \|$. On the other hand, since the norm on $\bigwedge V$ is submultiplicative, we have $\| \lambda \wedge \theta \| \leq \|\lambda\| \|\theta\| = 1$. Thus, $\|\lambda \wedge \theta\| = 1$. Furthermore, since
\begin{align*}
\|\lambda \wedge \theta - \phi\|^2 &= \|\lambda \wedge \theta\|^2 + \|\phi\|^2 - 2 \langle \lambda \wedge \theta, \phi \rangle \\
&= 1 + 1 - 2 = 0,
\end{align*}
we conclude that $\lambda \wedge \theta = \phi$.
\end{proof}

\begin{proof}[Proof of Proposition \ref{prop:exterior}]
The injectivity of $\delta_2$ is clear. For the second property in Proposition \ref{prop:x2relations}, we observe that, if $(\lambda, 1, \phi) \in X_2$ or $(1, \lambda, \phi) \in X_2$, then $\lambda = \phi$ by Lemma \ref{lemma:ext}.

In light of Lemma \ref{lemma:ext}, verification of the associativity property (Proposition \ref{prop:simplicialassociativity}) amounts to checking that, if $\lambda$, $\theta$, $\phi$, and $\lambda \wedge \theta \wedge \phi$ are all unit vectors, then $\lambda \wedge \theta$ is a unit vector if and only if $\theta \wedge \phi$ is a unit vector. This is true, since, by submultiplicativity, $\|\lambda \wedge \theta\| \leq \| \lambda \| \|\theta\| = 1$ and $1 = \|\lambda \wedge \theta \wedge \phi\| \leq \| \lambda \wedge \theta \| \|\phi\| = \|\lambda \wedge \theta\|$, so $\|\lambda \wedge \theta\| = 1$, and, similarly, $\|\theta \wedge \phi\| = 1$.

The rotation action of $\hat{\alpha} = \star$ takes $(\lambda, \theta, \phi)$ to $(\theta, \star \phi, \star \lambda)$. The invariance of $X_2$ follows from the fact that $\lambda \wedge \theta \wedge \star \phi = \theta \wedge \star \phi \wedge \star (\star \lambda)$, assuming that the degree of $\phi$ is the sum of the degrees of $\lambda$ and $\theta$.
\end{proof}

\subsection{Cohomology rings}

Let $M$ be a compact oriented Riemannian manifold of dimension $M$. Then the construction of Section \ref{sub:ext} can be applied fiberwise to the cotangent bundle $T^*M$. Furthermore, the construction completely passes to cohomology; in particular, the Hodge star passes to cohomology via its action on harmonic forms. Thus we obtain a Frobenius object in $\rel$ where $X$ is the unit sphere in $H^*(M)$.

In many cases, one can find smaller subsets $Y \subset H^*(M)$ that form sub-Frobenius objects. The requirements are as follows:
\begin{enumerate}
    \item $1 \in Y$,
    \item $Y$ is closed under $\star$,
    \item if $\lambda, \theta \in Y$ and $\|\lambda \wedge \theta\| = 1$, then $\lambda \wedge \theta \in Y$.
\end{enumerate}
We conclude with two simple examples.

\begin{example}
For any $M$, one can take $Y = \{1, \nu\}$. The Frobenius structure is given by
\begin{align*}
\eta & = \{1\},\\
\epsilon & = \{\nu\},\\
\mu & = \{(1,1,1), (1,\nu,\nu), (\nu, 1, \nu)\}, \\
\delta & = \{(1,\nu,1), (1,1,\nu), (\nu, \nu, \nu)\}.
\end{align*}
This is an example of a commutative Frobenius object.
\end{example}

\begin{example}\label{example:torus}
In the case where $M$ is the torus, one can take 
\[Y = \{\pm 1, \pm a, \pm b, \pm \nu\},\]
where $\{a,b\}$ is an orthonormal basis for $H^1(M)$. 

We note that, because the wedge product is only commutative up to sign, this is not a commutative Frobenius object. In fact, since $\hat{\alpha}^2(a) = \star^2 a = -a$, we see that, in this example, $\hat{\alpha}$ is not an involution. Thus we have an example of a Frobenius object with a nontrivial Nakayama automorphism.
\end{example}

\printbibliography
\end{document}